\documentclass[english,12pt,oneside]{amsproc}
\usepackage[english]{babel}
\usepackage{a4wide}
\usepackage{amsthm}
\usepackage{graphics}
\usepackage{amsfonts, amssymb, amscd, amsmath}
\usepackage{latexsym}
\usepackage[matrix,arrow,curve]{xy}
\usepackage{mathabx}%,mathtools}
\usepackage{color}
\usepackage{pbox}
\usepackage{tikz}
\usetikzlibrary{matrix,decorations.pathreplacing,positioning}
\usepackage{hyperref}

\DeclareMathOperator{\Cone}{Cone} 

 \DeclareMathOperator{\id}{id}

  \DeclareMathOperator{\ord}{Ord}
  
\DeclareMathOperator{\DAG}{DAG}

\DeclareMathOperator{\uDAG}{uDAG}
\DeclareMathOperator{\Pos}{Pos}
\DeclareMathOperator{\Stel}{Stel}
\DeclareMathOperator{\BStel}{BStel}
\DeclareMathOperator{\Cy}{Cy}
\DeclareMathOperator{\Jc}{Jc}

\DeclareMathOperator{\birth}{birth}

\DeclareMathOperator{\DisDS}{DisDS}
\DeclareMathOperator{\PrePos}{PrePos}
\DeclareMathOperator{\Top}{Top}
\DeclareMathOperator{\TopT0}{TopSep}

\newcommand{\Ro}{\mathbb{R}}
\newcommand{\Rg}{\mathbb{R}_{\geqslant 0}}

\newcommand{\Bb}{\mathbb{B}}

\newcommand{\ca}[1]{\mathcal{#1}}
\newcommand{\Hr}{\widetilde{H}}
\newcommand{\dd}{\partial}

\newcommand{\F}{\mathcal{F}}
\newcommand{\Pa}{\mathcal{P}}
\newcommand{\Ha}{\mathcal{H}}
\newcommand{\Pc}{\mathbf{Q}}

\newcommand{\Ss}{\mathbb{S}}
\newcommand{\bD}{\bar{D}}
\newcommand{\bP}{\bar{P}}

\newcounter{stmcounter}[section]
\newcounter{thcounter}

\numberwithin{equation}{section}

\theoremstyle{plain}
\newtheorem{cor}[stmcounter]{Corollary}

\newtheorem{thm}[thcounter]{Theorem}

\newtheorem{prop}[stmcounter]{Proposition}
\newtheorem{lem}[stmcounter]{Lemma}

\theoremstyle{definition}
\newtheorem{defin}[stmcounter]{Definition}

\theoremstyle{remark}
\newtheorem{ex}[stmcounter]{Example}
\newtheorem{rem}[stmcounter]{Remark}
\newtheorem{con}[stmcounter]{Construction}

\begin{document}

\title{Coverings by open and closed hemispheres}

\author{Anton Ayzenberg}
\address{Neapolis University Paphos, Cyprus, and Faculty of computer science, Higher School of Economics, Russia}
\email{ayzenberga@gmail.com}

\author{Maxim Beketov}
\address{Laboratory of algebraic topology and its applications, Faculty of computer science, Higher School of Economics, Russia}
\email{maxim.beketov@gmail.com }

\author{German Magai}
\address{Laboratory of algebraic topology and its applications, Faculty of computer science, Higher School of Economics, Russia}
\email{gmagaj@hse.ru}

\date{\today}
\thanks{The article was prepared within the framework of the HSE University Basic Research Program}

\subjclass[2020]{Primary 55P10, 55P15, 57Z25, 54A10, 52C35; Secondary 05C20, 52B35, 14N20, 05C40, 06A15, 52B12, 17B22, 52A55, 52B40.}

\keywords{Nerve theorem, covering by hemispheres, geometric lattice, space of directed acyclic graphs, strong connectivity of digraphs, Gale duality}

\begin{abstract}
In this paper we study the nerves of two types of coverings of a sphere $S^{d-1}$:
\begin{enumerate}
  \item coverings by open hemispheres;
  \item antipodal coverings by closed hemispheres.
\end{enumerate}
In the first case, nerve theorem implies that the nerve is homotopy equivalent to $S^{d-1}$. In the second case, we prove that the nerve is homotopy equivalent to a wedge of $(2d-2)$-dimensional spheres. The number of wedge summands equals the M\"{o}bius invariant of the geometric lattice (or hyperplane arrangement) associated with the covering. This result explains some observed large-scale phenomena in topological data analysis. We review the particular case, when the coverings are centered in the root system $A_d$. In this case the nerve of the covering by open hemispheres is the space of directed acyclic graphs (DAGs), and the nerve of the covering by closed hemispheres is the space of non-strongly connected directed graphs. The homotopy types of these spaces were described by Bj\"{o}rner and Welker, and the incarnation of these spaces appeared independently as ``the poset of orders'' and ``the poset of preorders'' respectively in the works of Bouc. We study the space of DAGs in terms of Gale and combinatorial Alexander dualities, and propose how this space can be applied in automated machine learning.
\end{abstract}

\maketitle

\section{Introduction}\label{secIntro}

Let $\Ss^{d-1}$ be the unit sphere in a space $V\cong\Ro^d$ centered in the origin. Consider a finite (multi)set of points $Z=\{z_1,\ldots,z_m\}\subset \Ss^{d-1}$ (repetitions are allowed); we call it a spherical configuration. Let us say that points $z_{i_1},\ldots,z_{i_k}$ form \emph{a constellation}, if they lie in a single open hemisphere. All constellations form a simplicial complex on the vertex set $[m]=\{1,\ldots,m\}$, we call it \emph{the constellation complex}:
\[
\Stel(Z)=\{\{i_1,\ldots,i_k\}\mid \text{ points }z_{i_1},\ldots,z_{i_k}\text{ lie in an open hemisphere}\}.
\]
Let $D(z)$ denote the open hemisphere centered in a point $z\in\Ss^{d-1}$. It can be seen that $\Stel(Z)$ coincides with the nerve of the covering which consists of open hemispheres $D(z_1),\ldots,D(z_m)$. We call a spherical configuration $Z$ \emph{ample} if any open hemisphere contains at least one point of $Z$. Equivalently, $Z$ is ample if the convex hull of the whole $Z$ contains the origin in its interior. Equivalently, $\Ss^{d-1}=\bigcup_{i\in[m]}D(z_i)$. The following statement is a direct consequence of the nerve theorem.

\begin{thm}\label{thmConstelSphere}
Assume $Z\subset\Ss^{d-1}$ is ample. Then the constellation complex $\Stel(Z)$ is homotopy equivalent to $S^{d-1}$.
\end{thm}

Indeed, the intersections of open hemispheres are contractible unless empty therefore the nerve theorem is applicable in this case. The construction becomes more interesting if one replaces open hemispheres by closed hemispheres. We say that points $z_{i_1},\ldots,z_{i_k}$ form \emph{a big constellation}, if they lie in a single closed hemisphere. The simplicial complex of all big constellations is called \emph{the big constellation complex}:
\[
\BStel(Z)=\{\{i_1,\ldots,i_k\}\mid \text{ points }z_{i_1},\ldots,z_{i_k}\text{ lie in a closed hemisphere}\}.
\]
If $\bD(z)$ denotes the closed hemisphere centered in $z\in\Ss^{d-1}$, then $\BStel(Z)$ is the nerve of the covering of $\Ss^{d-1}$ by the closed subsets $\bD(z_1),\ldots,\bD(z_m)$. Notice that the nerve theorem is not applicable in this case, since intersections of closed hemispheres may be non-contractible. For example, the intersection of two antipodal closed hemispheres is an equatorial subsphere $S^{d-2}$.

In general, we do not know the homotopy type of $\BStel(Z)$. However, the homotopy type can be described in a particular case. We say that a spherical configuration $Z$ is \emph{antipodal}, if, whenever, $z\in Z$, the antipodal point $-z\in\Ss^{d-1}$ is also contained in $Z$ (with the same multiplicity).

\begin{thm}\label{thmBigConstelSphere}
Assume that $Z\subset\Ss^{d-1}$ is ample and antipodal. Then the big constellation complex $\BStel(Z)$ is homotopy equivalent to a wedge of spheres of dimension $2(d-1)$.
\end{thm}

To specify the number of wedge summands, notice that each pair of antipodal points $\{z,-z\}$ on a sphere determines a unique line $l\subset V\cong\Ro^d$ through the origin. The orthogonal complement $l^\bot$ is a hyperplane. Hence antipodal configurations with $2r$ points correspond uniquely to hyperplane arrangements $\ca{H}=\{H_1,\ldots,H_r\}$ in $V$ (where repetitions of hyperplanes are allowed). Each hyperplane arrangement $\ca{H}$ determines the geometric lattice $S(\ca{H})$ of all possible intersections of hyperplanes in the arrangement. This lattice has rank $d$ (provided that $Z$ is ample); it has the least element $\hat{0}$ (the subspace $\{0\}$) and the greatest element $\hat{1}$ (the space $V$ itself). The value $\mu(\ca{H})=\mu_{S(\ca{H})}(\hat{0},\hat{1})$ of the M\"{o}bius function is called \emph{the M\"{o}bius invariant} of the configuration $\ca{H}$. According to~\cite{BjMatr}, the geometrical realization $|S(\ca{H})|$ is homotopy equivalent to the wedge $\bigvee_{\mu(\ca{H})}S^{d-2}$.

The number of wedge summands in Theorem~\ref{thmBigConstelSphere} equals $\mu(\ca{H})$ for the hyperplane arrangement $\ca{H}$ corresponding to the antipodal spherical configuration $Z$.

Our proof of Theorem~\ref{thmBigConstelSphere} is essentially based on the techniques developed in~\cite{BjWel} and extends the main result of that paper in the following sense.

Consider the simplicial complex $\DAG_n$ encoding all directed acyclic graphs on a fixed vertex set $X$ of cardinality $n$. Also consider the simplicial complex $\DisDS_n$ encoding the property of a directed graph to be non-strongly connected, see details in Section~\ref{secDigraphs}. Then the spaces $\DAG_n$ and $\DisDS_n$ appear as a natural example of a constellation complex and a big constellation complex respectively. More precisely, consider the vector space $\Ro^n$ with a fixed basis $e_1,\ldots,e_n$ and consider the collection of points
\begin{equation}\label{eqAsystem}
A_{n-1}=\{\alpha_{ij}=e_j-e_i\mid (i,j) \text{ is an ordered pair in }[n]\}.
\end{equation}
This is basically the root system of type $A$. The points lie in a hyperplane $\Pi=\{\sum x_i=0\}\cong\Ro^{n-1}$, and after suitable normalization we may assume that $A_{n-1}\subset \Ss^{n-2}$. The spherical configuration $A_{n-1}$ is ample and antipodal, the corresponding hyperplane arrangement is the classical hyperplane arrangement of type $A$. We notice the following simple fact.

\begin{prop}\label{propGraphsConstellations}
The following simplicial complexes are naturally isomorphic
\[
\DAG_n=\Stel(A_{n-1}),\qquad \DisDS_n=\BStel(A_{n-1}).
\]
\end{prop}

The M\"{o}bius invariant of the hyperplane configuration of type $A$ in $\Ro^n$ equals $(n-1)!$, see~\cite{Vas}. Then, from Theorems~\ref{thmConstelSphere} and~\ref{thmBigConstelSphere} and Proposition~\ref{propGraphsConstellations}, it follows that
\begin{equation}\label{eqHomotopiesGraphs}
\DAG_n\simeq S^{n-2},\qquad \DisDS_n\simeq\bigvee\nolimits_{(n-1)!}S^{2n-4}.
\end{equation}
These two homotopy equivalences are the main result of~\cite{BjWel}.

We observe the following connection, explained in detail in Section~\ref{secOrdOrdersTopologies}. There exists a natural homotopy equivalence, via Galois connection, between the simplicial complex $\DAG_n$ and the poset $\Pos_n$ of all (nontrivial) orders on a given set of cardinality $n$. Similarly, there exists a natural homotopy equivalence between the simplicial complex $\DisDS_n$ and the poset $\PrePos_n$ of all (nontrivial) preorders on a given set of cardinality~$n$. Therefore, we have
\begin{equation}\label{eqHomotopiesGraphs}
|\Pos_n\setminus\{\hat{0},\hat{1}\}|\simeq S^{n-2}\qquad |\PrePos_n\setminus\{\hat{0},\hat{1}\}|\simeq\bigvee\nolimits_{(n-1)!}S^{2n-4},
\end{equation}
which was found by Bouc (for the first equivalence see~\cite{Bouc}, the second is formulated in~\cite{BoucTalk}). Notice that preorders on $X$ correspond bijectively to topologies on $X$, and orders on $X$ correspond to topologies satisfying Kolmogorov $T_0$-separability axiom. The inclusion of (pre)orders naturally corresponds to increasing the strength of topology. If $\Top(X)$ denotes the poset of all topologies on $X$ ordered by strength, and $\TopT0(X)$ is the poset of all $T_0$-topologies, then we get
\begin{equation}\label{eqHomotopiesTopologies}
|\Top(X)\setminus\{\hat{0},\hat{1}\}|\simeq \bigvee_{(n-1)!}S^{2n-4} \qquad |\TopT0(X)\setminus\{\hat{1}\}|\simeq S^{n-2},
\end{equation}
where $\hat{1}$ is the discrete topology on $X$, and $\hat{0}$ is indiscrete topology (the latter is not $T_0$ so there is no need to remove it from $\TopT0(X)$).

In Section~\ref{secGale} we recall a simple observation made by the first author in~\cite{AyzConst}, that constellation complexes appear naturally as combinatorial Alexander dual complexes to nerve complexes of convex polytopes. This observation utilizes the basic properties of the Gale duality. From this point of view, the ``spherical'' complex $\DAG_n=\Stel(A_{n-1})$ can be treated as the dual object to some (generally, non-simple and non-simplicial) convex polytope $\Pc_n$ of dimension $n(n-2)$. This polytope, in a certain sense, encodes all possible cycles on $n$ vertices, see Proposition~\ref{propFacesCycles} for the precise statement.

In the last section we provide preliminary constructions which hopefully pave the way to using the space of directed acyclic graphs (DAGs) in automated machine learning.

\section{Proof of the main theorem}\label{secBigConstel}

\subsection{Preliminaries}\label{subsecQuillen}

Quillen's theorem A for posets appears quite often in our argument, so we recall this result and related statements for convenience of the reader.

A morphism of posets $f\colon S\to T$ is a monotone map, that is $s_1\leqslant_S s_2$ implies $f(s_1)\leqslant_T f(s_2)$. Given two morphisms $f\colon S\to T$ and $g\colon T\to S$ we write them as $f\colon S\rightleftarrows T\colon g$. If $s$ is an element of a poset $S$, then $S_{\geqslant s}$ denotes the subposet $\{t\in S\mid t\geqslant s\}$. The subposets $S_{\leqslant s}$, $S_{>s}$, $S_{<s}$ are defined in a similar fashion.

Any (finite) poset $S$ can be transformed to a CW-complex in a functorial way as follows. Consider the simplicial complex $\ord(S)$ called \emph{the order complex} of $S$, whose vertices are elements of $S$ and simplices are given by chains (linearly ordered subposets) in $S$. The geometrical realization of $\ord(S)$ is called \emph{the geometrical realization} of $S$ and denoted by $|S|$. Notice that any morphism $f\colon S\to T$ induces a cellular map $f_*\colon|S|\to|T|$.

\begin{thm}[Quillen's theorem A or Quillen's fiber theorem~\cite{Quil}]\label{thmQuilMcCord}
Let $f\colon S\to T$ be a morphism of finite partially ordered sets. Suppose that, for any $t\in T$, the geometrical realization  $|f^{-1}(T_{\geqslant t})|$ is contractible. Then $f$ induces a homotopy equivalence between $|S|$ and $|T|$.
\end{thm}

We refer to~\cite{Barmak} for a modern exposition of this result, and to~\cite{Bj} for related statements. Sometimes this theorem is called Quillen--McCord theorem, due to its relation to finite topologies studied by McCord~\cite{McCord}.

\begin{defin}\label{defGaloisCon}
A pair of morphisms $f\colon S\rightleftarrows T\colon g$ is called \emph{a(n order preserving) Galois connection}, if one of the two equivalent conditions hold true:
\begin{enumerate}
  \item For any $s\in S$ and $t\in T$ the condition $f(s)\geqslant_T t$ is equivalent to $s\geqslant_S g(t)$.
  \item $g(f(s))\leqslant_S s$ for any $s\in S$ and $f(g(t))\geqslant_T t$ for any $t\in T$.
\end{enumerate}
\end{defin}

Treating posets as small categories, and morphisms as functors, Galois connection becomes a
The next statement can be deduced from Theorem~\ref{thmQuilMcCord} or proved independently by homotopy theoretical reasoning.

\begin{cor}\label{corGalois}
In a Galois connection $f\colon S\rightleftarrows T\colon g$, both maps $f$ and $g$ induce homotopy equivalence between the geometrical realizations $|S|$ and $|T|$.
\end{cor}

\begin{proof}
By the definition of Galois connection we have
\[
f^{-1}(T_{\geqslant t})=\{s\in S\mid f(s)\geqslant t\} = \{s\in S\mid s\geqslant g(t)\}=S_{\geqslant g(t)}.
\]
The latter poset has the least element $g(t)$, hence its geometrical realization is a cone, hence contractible. Then Quillen--McCord theorem applies. The proof for $g$ is completely similar.
\end{proof}

\begin{rem}
A slightly modified argument shows that the maps $|f|\colon|S|\rightleftarrows|T|\colon |g|$ actually form a homotopy eqiuvalence (that is $|f\circ g|\simeq\id_{|T|}$ and $|g\circ f|\simeq\id_{|S|}$).
\end{rem}

\subsection{Reduction to the geometric lattice}\label{subsecGeomLattice}

Our proof of Theorem~\ref{thmBigConstelSphere} is the extended version of the arguments of~\cite{BjWel} where this result was proved for the particular case $\BStel(A_{n-1})$.

At first we introduce more convenient notation for the objects under consideration. To simplify exposition, we will identify vector space $V\cong\Ro^d$ with its dual $V^*$ by choosing some inner product.

\begin{con}\label{conGeomLatticeMobius}
Let $\Ha=\{H_1,\ldots,H_r\}$ be a linear hyperplane arrangement in $V\cong\Ro^d$, that is a collection of linear hyperplanes (repetitions are allowed). Let $S(\Ha)$ be the set of all possible intersections of $H_i$'s, that is
\[
S(\Ha)=\left\{\bigcap\nolimits_{i\in I}H_i\mid I\subseteq [r]\right\},
\]
where we formally set $\bigcap_{i\in \varnothing}H_i=V$. Repetitions are not allowed in $S(\Ha)$. The set $S(\Ha)$ is partially ordered by inclusion. The greatest element of $S(H)$ is the space $V$ itself, we denote it by $\hat{1}$. The least element $\hat{0}$ of $S(H)$ is the intersection of all hyperplanes $\bigcap_{i\in[r]}H_i$.

It will be assumed in the following that
\begin{equation}\label{eqAssumptionGeneral}
\hat{0}=\{0\},
\end{equation}
which is equivalent to saying that normals to $H_i$'s linearly span $V$. The poset $S(\Ha)$ is a geometric lattice, its elements are graded by dimensions of vector subspaces. The rank of the lattice $S(\Ha)$ equals $d=\dim V$, if assumption~\eqref{eqAssumptionGeneral} holds true. Let us use the notation $\bar{S}(\Ha)=S(\Ha)\setminus\{\hat{0},\hat{1}\}$.

We also need the dual geometric lattice $S(\Ha)^*$. On the abstract level, it is obtained from $S(\Ha)$ by reversing the order. On the other hand, the elements of $S(\Ha)^*$ can be naturally identified with orthogonal complements to elements of $S(\Ha)$. The symbol $\bar{S}(\Ha)^*$ denotes the poset $S(\Ha)^*\setminus\{\hat{0},\hat{1}\}$ of proper elements of the dual lattice.
\end{con}

Let $\mu(\Ha)$ denote the value of the M\"{o}bius function $\mu_{S(\Ha)}(\hat{0},\hat{1})$. The following is the result of Bj\"{o}rner~\cite{BjMatr}, and its homological version was previously proven by Folkman in~\cite{Folk}.

\begin{thm}[{\cite[Thm.7.9.1]{BjMatr}, also see~\cite[Thm.4.1]{Folk}}]\label{thmBjorner}
The geometrical realization of the poset $\bar{S}(\Ha)$ (and hence $\bar{S}(\Ha)^*$) is shellable. It is homotopy equivalent to the wedge of $\mu(\Ha)$ many $(d-2)$-dimensional spheres:
\begin{equation}\label{eqGeomLatticeWedge}
|\bar{S}(\Ha)^*|=|\bar{S}(\Ha)|\simeq\bigvee\nolimits_{\mu(\Ha)}S^{d-2}.
\end{equation}
\end{thm}

Let us return back to collections of points on the unit sphere.

\begin{con}\label{conSphericalFromHyperplanes}
Given a hyperplane arrangement $\Ha=\{H_1,\ldots,H_r\}$, consider the collection of points on the unit sphere
\begin{equation}\label{eqZh}
Z_\Ha=\{z_1,z_{-1},\ldots,z_{r},z_{-r}\}\subset\Ss^{d-1}\subset V,
\end{equation}
where $z_i,z_{-i}$ are the unit normals to the hyperplane $H_i$ for any $i\in[r]$. Repetitions are allowed in $Z_\Ha$. By construction, $Z_\Ha$ is antipodal, and assumption~\eqref{eqAssumptionGeneral} implies that $Z_\Ha$ is ample. It is easily seen that, conversely, any antipodal ample configuration on a sphere is induced by some hyperplane arrangement $\Ha$.
\end{con}

Now let $Z=\{z_1,\ldots,z_m\}$ be a spherical configuration, and $I\subseteq[m]$ be a subset of indices. Consider the convex cone generated by $z_i$ for $i\in I$:
\[
C_I=\Cone\{z_i\mid i\in I\}=\sum_{i\in I}\Rg z_i\subseteq V.
\]
The conditions that a collection of vectors lies in an open or a closed hemisphere can be read from the properties of $C_I$. More precisely, the following hold:
\begin{enumerate}
  \item $I\in \Stel(Z)$ (the vectors $\{z_i\mid i\in I\}$ lie in an open hemisphere) if and only if $C_I$ is a strictly convex cone.
  \item $I\in \BStel(Z)$ (the vectors $\{z_i\mid i\in I\}$ lie in a closed hemisphere) if and only if $C_I\neq V$.
\end{enumerate}
Recall that a cone is called strictly convex if it does not contain a line. It will be convenient for us to introduce some notation which allows to interpolate between items 1 and 2 above.

\begin{con}\label{conEdgeOfCone}
Let $W_I$ denote the greatest (by inclusion) vector subspace of $V$ contained in the cone $C_I$. In other words $W_I=C_I\cap(-C_I)$. We call $W_I$ \emph{the ridge} of $C_I$. The ridge of a strictly convex cone is a single point $\{0\}$. It follows that
\begin{enumerate}
  \item $I\in \Stel(Z)$ if and only if $W_I=\{0\}$.
  \item $I\in \BStel(Z)$ if and only if $W_I\neq V$.
\end{enumerate}
The dimensions of ridges $W_I$, for various $I\subset[m]$, may vary between $0$ and $d$. We will extensively use this freedom in the subsequent arguments.
\end{con}

\begin{lem}\label{lemLandsInSH}
For any $I\in \BStel(Z_\Ha)$, the ridge $W_I$ belongs to $S(\Ha)^*$.
\end{lem}

\begin{proof}
Consider the subset $J=\{i\in I\mid z_i\in W_I\}$ which indexes all generators of $C_I$ lying in the ridge. By construction, nonnegative combinations of $\{z_i\mid i\in J\}$ span the subspace $W_I$. It follows in particular that $W_I=\sum_{i\in J} \langle z_i\rangle$. Hence
\[
W_I^\bot=\bigcap_{i\in J}\langle z_i\rangle^\bot=\bigcap_{i\in J}H_i\in S(\Ha).
\]
Therefore $W_I\in S(\Ha)^*$.
\end{proof}

Recall the following construction on posets.

\begin{con}\label{conPlusPosets}
Let $S_1$, $S_2$ be posets with order relations $\leqslant_1$ and $\leqslant_2$ respectively. \emph{The ordered sum} $S_1+S_2$ is a disjoint union $S_1\sqcup S_2$ endowed with the partial order $\leqslant$ such that $s_1\leqslant s_2$ for any $s_1\in S_1$ and $s_2\in S_2$, and $\leqslant$ coincides with $\leqslant_1$ and $\leqslant_2$ on the respective components.

Any chain in $S_1+S_2$ is a concatenation of a chain in $S_1$ and a chain in $S_2$. Therefore, from the definition of the geometrical realization of a poset, it follows that
\begin{equation}\label{eqPlusJoin}
|S_1+S_2|\cong |S_1|\ast|S_2|,
\end{equation}
where $\ast$ is the join operation of topological spaces.
\end{con}

Now we construct a specific map of posets needed for the proof of Theorem~\ref{thmBigConstelSphere}.

\begin{con}\label{conMainMap}
Recall that $\Ha=\{H_1,\ldots,H_r\}$ is a hyperplane arrangement in $V\cong\Ro^d$, $Z_\Ha$ is the corresponding spherical configuration, $\Stel(Z_\Ha)\subset\BStel(Z_\Ha)$ are the constellation complex and the big constellation complex respectively. Both complexes have vertex set $[2r]=\{1,-1,\ldots,r,-r\}$. In the following we consider these complexes as posets, and do not take the empty simplex into account.

Consider the map
\begin{equation}\label{eqMainMap}
f\colon \BStel(Z_\Ha)\to \Stel(Z_\Ha)+\bar{S}(\Ha)^*
\end{equation}
defined by
\[
  f(I)=\begin{cases}
         I, & \mbox{if } I\in \Stel(Z_\Ha) \\
         W_I, & \mbox{otherwise}.
       \end{cases}
\]
Notice that $W_I$ is always an element of $S(\Ha)^*$ by Lemma~\ref{lemLandsInSH}; for $I\in \BStel(Z_\Ha)$ the value $W_I$ is never the element $V=\hat{1}\in S(\Ha)^*$ (Construction~\ref{conEdgeOfCone}), and $W_I\neq \hat{0}\in S(\Ha)^*$ (since in this case $I\in\Stel(Z_\Ha)$ and the map lands in the first summand). Therefore the map $f$ is well defined. It is easily seen that this map is monotonic.
\end{con}

We assert that the map $f$ defined by~\eqref{eqMainMap} satisfies the assumption of Quillen's fiber theorem, more precisely, the assumption with the reversed order.

\begin{thm}\label{thmApplyQuillen}
The map $f\colon \BStel(Z_\Ha)\to T=\Stel(Z_\Ha)+\bar{S}(\Ha)^*$ defined in Construction~\ref{conMainMap} satisfies the property: for any element $t\in T$ the geometrical realization of the preimage $f^{-1}(T_{\leqslant t})$ is contractible.
\end{thm}

We postpone the proof to the next subsection and concentrate on the assertion of Quillen's theorem.

\begin{cor}
The map $f$ induces the homotopy equivalence
\[
|\BStel(Z_\Ha)| \simeq |\Stel(Z_\Ha)+\bar{S}(\Ha)^*|\cong |\Stel(Z_\Ha)|\ast |\bar{S}(\Ha)^*|.
\]
\end{cor}

Remembering that $|\Stel(Z_\Ha)|\simeq S^{d-1}$ by Theorem~\ref{thmConstelSphere} and $|\bar{S}(\Ha)^*|\simeq \bigvee_{\mu(\Ha)}S^{d-2}$ by Theorem~\ref{thmBjorner}, we obtain the homotopy equivalence
\[
|\BStel(Z_\Ha)|\simeq S^{d-1}\ast\left(\bigvee\nolimits_{\mu(\Ha)}S^{d-2}\right)\simeq \bigvee\nolimits_{\mu(\Ha)}S^{2d-2},
\]
which proves Theorem~\ref{thmBigConstelSphere}.

\subsection{Proof of Theorem~\ref{thmApplyQuillen}}\label{subsecMainProof}

Since the target poset $T$ of the map $f\colon \BStel(Z_\Ha)\to T$ is the ordered sum $\Stel(Z_\Ha)+\bar{S}(\Ha)^*$, the element $t\in T$ is either an element of $\Stel(Z_\Ha)$ or an element of $\bar{S}(\Ha)^*$. These two options are considered separately.

\textbf{Case 1.} $t\in \Stel(Z_\Ha)$. We rename $t$ by $I$, it is a nonempty simplex of the constellation complex. Since $\Stel(Z_\Ha)\subset\BStel(Z_{\Ha})$ and $f$ acts identically on $\Stel(Z_\Ha)$, we have
\[
f^{-1}(T_{\leqslant I})=f^{-1}(\Stel(Z_\Ha)_{\leqslant I})=\BStel(Z_\Ha)_{\leqslant I}
\]
The latter poset has the greatest element, the simplex $I$ itself, hence its geometrical realization is contractible.

\textbf{Case 2.} $t\in \bar{S}(\Ha)^*$. Again, for the sake of soundness, rename $t$ by $\Pi$, this is an element of $\bar{S}(\Ha)^*$, hence a proper vector subspace of $V$. By construction of the map $f$, the subposet
\[
K_\Pi=f^{-1}(T_{\leqslant \Pi})
\]
consists of all simplices $I\in \BStel(Z_\Ha)$ such that ridge $W_I$ of the cone $C_I=\Cone\{z_i\mid i\in I\}$ is contained in the flat $\Pi$. Since $I_1\subseteq I_2$ implies $W_{I_1}\subseteq W_{I_2}$, the subset $K_\Pi$ is a simplicial subcomplex of $\BStel(Z_\Ha)$ (this also follows from the monotonicity of the map $f$).

We need to prove that $K_\Pi$ is contractible. The strategy will be the following: we realize $K_\Pi$ as a nerve of a certain covering of the whole vector space $V$ by closed convex sets, and deduce contractibility of $K_\Pi$ from contractibility of $V\cong\Ro^d$ via the nerve theorem. To realize this strategy, several new constructions are required.

\begin{con}
Let us fix a flat $\Pi\in \bar{S}(\Ha)^*$. For $i\in[2r]$ consider the open and closed halfspaces of $V$ determined by the point $z_i\in \Ss^{d-1}$:
\[
P_i=\{u\in V\mid \langle u, z_i\rangle>0\},\quad \bP_i=\{u\in V\mid \langle u, z_i\rangle\geqslant0\}.
\]
Consider the collection of sets $\ca{A}_\Pi=\{A_i\mid i\in[2r]\}$,
\[
A_i=\begin{cases}
      \bP_i, & \mbox{if } z_i\in \Pi \\
      P_i, & \mbox{if } z_i\notin\Pi.
    \end{cases}
\]
\end{con}

\begin{prop}\label{propFiberNerveHalves}
The simplicial complex $K_\Pi$ coincides with the nerve $N(\ca{A}_\Pi)$ of the collection $\ca{A}_\Pi$. In other words, $I\in K_\Pi$ if and only if $\bigcap_{i\in I}A_i\neq\varnothing$.
\end{prop}

\begin{proof}
At first, we prove an inclusion $K_\Pi\subseteq N(\ca{A}_\Pi)$. Let $I\in K_\Pi$ so that $W_I\subseteq \Pi$. The image $C_I/W_I$ of the cone $C_I$ under the natural orthogonal projection $p\colon V\to V/W_I\cong W_I^\bot$ is a strictly convex cone. Therefore, there exists $u\in W_I^\bot$ such that $\langle u, p(z_i)\rangle>0$ for any $z_i\notin W_I$. On the other hand, since $u\in W_I^\bot$, we have $\langle u, p(z_i)\rangle=0$, consequently $\langle u, p(z_i)\rangle\geqslant 0$. Since $W_I\subseteq\Pi$, these inequalities imply
\begin{equation}\label{eqIneqs2types}
\begin{cases}
  \langle u, z_i\rangle\geqslant 0, & \mbox{if } i\in I \mbox{ and } z_i\in \Pi \\
  \langle u, z_i\rangle>0, & \mbox{if } i\in I \mbox{ and } z_i\notin \Pi.
\end{cases}
\end{equation}
This shows that $u$ lies in $\bigcap_{i\in I}A_i$, so this intersection is nonempty. Therefore $I\in N(\ca{A}_\Pi)$.

The rest of the proof is devoted to an inclusion $N(\ca{A}_\Pi)\subseteq K_\Pi$. Let $I\in N(\ca{A}_\Pi)$ so the intersection $\bigcap_{i\in I}A_i$ is nonempty. Pick an element $u\in \bigcap_{i\in I}A_i$. For this element, inequalities~\eqref{eqIneqs2types} hold. As before, consider the cone $C_I$ and the maximal vector subspace $W_I$ of $C_I$. Consider the subset $J\subseteq I$ which labels all vectors of $\{z_i\mid i\in I\}$ which lie in $W_I$. Since $W_I$ is nonnegatively spanned by $\{z_j\mid j\in J\}$, there exists a linear relation of the form
\begin{equation}\label{eqLinRelZero}
\sum_{j\in J}a_jz_j=0,
\end{equation}
where all coefficients $a_j>0$. Indeed, the origin lies in the relative interior of the convex hull of the set $\{z_j\mid j\in J\}$. However this relative interior coincides with the set
\[
\left\{\sum\nolimits_{j\in J}\lambda_jz_j\left| \sum\nolimits_{j\in J}\lambda_j=1, \text{ and } \lambda_j>0 \text{ for any } j\in J\right.\right\},
\]
since the latter subset is open in $W_I$ and its closure coincides with the convex hull.

Taking scalar product of~\eqref{eqLinRelZero} with $u$ we obtain
\begin{equation}\label{eqScalarProdRel}
0=\langle u,0\rangle=\sum_{j\in J} a_j\langle u, z_j\rangle.
\end{equation}
If at least one summand $\langle u, z_j\rangle$ is strictly positive, then the whole expression on the right hand side of~\eqref{eqScalarProdRel} is positive. This observation proves $\langle u, z_j\rangle=0$ for any $z_j\in W_I$. Then any such $z_j$ belongs to $\Pi$, since otherwise we get a contradiction with~\eqref{eqIneqs2types}. Therefore $W_I\subseteq\Pi$ and hence $I\in K_\Pi$.
\end{proof}

Strictly speaking we cannot apply the Nerve theorem directly to the collection $\ca{A}_\Pi$ because of its mixedness: some subsets in this collection are closed and some are open. To remedy this, we consider another collection $\ca{B}_\Pi=\{B_i\mid i\in[2r]\}$ defined by
\[
B_i=\begin{cases}
      \bP_i=\{u\mid \langle u,z_i\rangle\geqslant 0\}, & \mbox{if } z_i\in \Pi \\
      \{u\mid \langle u,z_i\rangle\geqslant 1\}, & \mbox{if } z_i\notin\Pi.
    \end{cases}
\]

\begin{lem}\label{lemTwoNerves}
The nerves of the collections $\ca{A}_\Pi$ and $\ca{B}_\Pi$ coincide:
\[
N(\ca{A}_\Pi)=N(\ca{B}_\Pi).
\]
\end{lem}

\begin{proof}
It is straightforward that $\bigcap_{i\in I}B_i\neq\varnothing$ implies $\bigcap_{i\in I}A_i\neq\varnothing$ since $B_i\subseteq A_i$ for any $i$. On the other hand, if $u\in \bigcap_{i\in I}A_i$, then $\alpha u\in \bigcap_{i\in I}B_i$ for sufficiently large $\alpha>0$. For example $\alpha=1/\min\{\langle u, z_i\rangle\mid z_i\notin \Pi\}$ is sufficient for this observation.
\end{proof}

%\begin{rem}
%If the reader
%\end{rem}

The nerve theorem applied to the closed covering $\ca{B}_\Pi$ implies that $K_\Pi=N(\ca{B}_\Pi)$ is homotopy equivalent to the union $\bigcup_{i\in[2r]}B_i$. It remains to check that this union is indeed contractible.

\begin{lem}\label{lemUnionIsAll}
The union $\bigcup_{i\in[2r]}B_i$ is the whole space $V$.
\end{lem}

\begin{proof}
Recall that $\Pi\in S(\Ha)^*\setminus\{\hat{0},\hat{1}\}$, so it contains a line from $S(\Ha)^*$, and consecutively, some pair of antipodal vectors $z_i$ and $z_{-i}=-z_i$ from $Z_\Ha$. The closed subsets $B_i=\bP_i=\{\langle u, z_i\rangle\geqslant 0\}$ and $B_{-i}=\bP_{-i}=\{\langle u, -z_i\rangle\geqslant 0\}$ are a pair of complementary closed halfspaces. The union of these two subsets is already the whole space $V$.
\end{proof}

\begin{rem}
A reader could have noticed that the proof of Lemma~\ref{lemUnionIsAll} is the first and the only place in the arguments where the antipodality of the spherical configuration is essentially used. However, the fact that both $z$ and $-z$ are present in the configuration is crucial in this point. If some lines $\Pi$ from $S(\Ha)$ contain only one of the antipodal vectors, then the union $\bigcup_{i\in[m]}B_i$ (and hence the fiber $K_\Pi=f^{-1}(T_{\leqslant\Pi})$) may be non-contractible. Fig.~\ref{figHalfspaces} shows a simple example when such situation occurs.
\end{rem}

\begin{figure}[h]
\begin{center}
\includegraphics[scale=0.4]{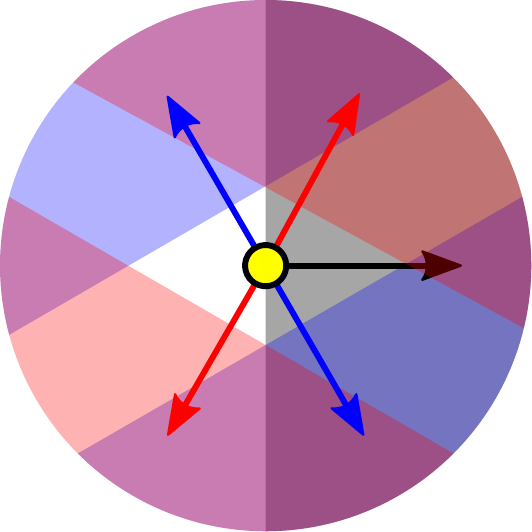}
\end{center}
\caption{One line contains only one vector $z$. The union of closed halfplanes $B_i$ corresponding to the five vectors is homotopy equivalent to the circle $S^1$. This example shows that antipodality of a spherical configuration is crucial in our proof of Theorem~\ref{thmBigConstelSphere}.}\label{figHalfspaces}
\end{figure}

\section{Application to digraphs}\label{secDigraphs}

For a finite set $X$ consider the set $X^{(2)}$ of ordered pairs of distinct elements:
\[
X^{(2)}=\{(x,y)\mid x,y\in X, x\neq y\}.
\]
A directed graph on a vertex-set $X$ is completely determined by a subset $E\subseteq X^{(2)}$, the set of edges. Therefore digraphs on the vertex set $X$ are in one to one correspondents with subsets of~$X^{(2)}$. A property of digraphs on $X$ can be identified with the set of digraphs which have this property, in other words, with collections of subsets of $X^{(2)}$. A property $\Pa$ is called monotonic (downwards) if it is inherited by subgraphs. Monotonic properties correspond to simplicial complexes on the vertex set $X^{(2)}$.

\begin{con}\label{conDefDAG}
A directed graph is called \emph{acyclic} if it does not have directed cycles. Let $\DAG_n$ denote the property of a directed graph on $n$ vertices to be acyclic. The same notation is used for the corresponding simplicial complex on the vertex set $[n]^{(2)}$ of all ordered pairs of indices. The simplicial complex $\DAG_n$ (more precisely, its geometrical realization) can be treated as the space of all directed acyclic graphs on $n$ vertices. This interpretation is explained in more detail in Section~\ref{secSpherDAGs}. The example of the simplicial complex $\DAG_3$ is shown on Fig.~\ref{figDAG3}.
\end{con}

\begin{figure}[h]
\begin{center}
\includegraphics[scale=1]{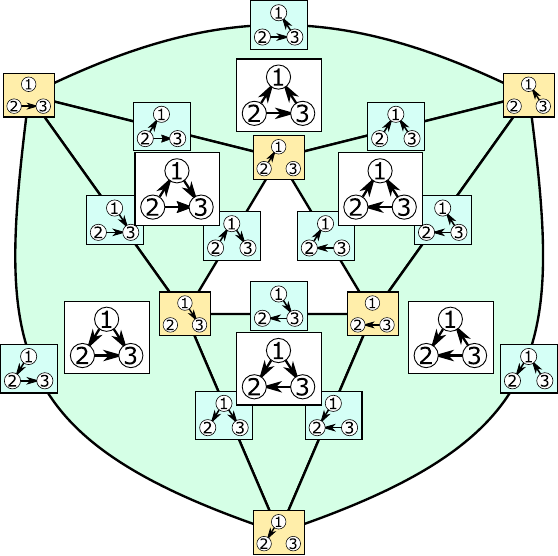}
\end{center}
\caption{The simplicial complex $\DAG_3$ of directed acyclic graphs on $3$ vertices.}\label{figDAG3}
\end{figure}

\begin{con}\label{conDefDisconS}
A directed graph is called \emph{strongly connected} if, for any two vertices $i,j$ there exists a directed path from $i$ to $j$. The property of being strongly connected is monotonic upwards, in the sense that given a strongly connected graph, any bigger graph on the same vertex set is strongly connected as well. Since we want to work with properties monotonic downwards, we consider the negation of strongly connectedness. The property of a graph on $n$ vertices to be not strongly connected as well as the corresponding simplicial complex is denoted $\DisDS_n$. Again, the geometrical realization of this simplicial complex can be treated as the space of all digraphs which fail to be strongly connected. Compare this construction with the space of disconnected undirected graphs studied by Vassiliev in~\cite{Vas}. We have a straightforward inclusion $\DAG_n\subset \DisDS_n$ for $n\geqslant2$, since acyclic digraphs can be alternatively described as the digraphs in which all components of strong connectivity are singletons.
\end{con}

Before studying homotopy types of $\DAG_n$ and $\DisDS_n$, let us indicate some basic properties of these simplicial complexes. We are especially interested in the space $\DAG_n$ of all acyclic digraphs.

Recall the standard notions from the theory of simplicial complexes. Let $K$ be a(n abstract) simplicial complex on a vertex set $V$. For a simplex $I\in K$ the number $\dim I=|I|-1$ is called \emph{the dimension} of a simplex, and $\dim K=\max_{I\in K}\dim I$. A simplex $I$ is called \emph{maximal} if there is no $J\in K$ such that $J\supsetneq I$. A simplicial complex is called \emph{pure} if all its maximal simplices have the same dimension. In a pure simplicial complex, maximal simplices are also called \emph{facets}, and codimension 1 simplices are called \emph{ridges}. A pure simplicial complex $K$ is called \emph{a pseudomanifold} if any ridge is contained in exactly 2 facets. A pure simplicial complex $K$ is called \emph{a pseudomanifold with boundary} if any ridge is contained in either 1 or 2 facets. In this case the union of ridges which are contained in one facet is called \emph{the boundary} of $K$.

\begin{lem}\label{lemPseudoMfd}
The simplicial complex $\DAG_n$ is pure of dimension $(n+1)(n-2)/2$ with $n!$ facets. It is a pseudomanifold with boundary.
\end{lem}

\begin{proof}
By a total DAG on $X$ we mean the directed graph of any total order on $X$. Since there exist $n!$ many total orders on $X$, we have $n!$ total DAGs. All total DAGs have $n(n-1)/2$ directed edges. Any acyclic directed graph is contained in some total DAG, since any partial order can be extended to a total order (such extensions are also called topological sortings in computer science literature). Therefore total DAGs are the maximal simplices of $\DAG_n$, they all have dimension $n(n-1)/2-1=(n+1)(n-2)/2$. This proves the first part of the statement.

Notice that any ridge $J$ of $\DAG_n$ is obtained from some total DAG $I$ by removing one oriented edge, say $(u,v)$. For any pair $\{x,y\}\neq\{u,v\}$ exactly one of the ordered pairs, $(x,y)$ or $(y,x)$, is an oriented edge of $J$. If e.g. $(x,y)$ is an edge, then adding $(y,x)$ to $J$ creates an oriented cycle. This shows that the ridge $J$ is contained in at most two candidate facets: $I$ and $I'=J\sqcup\{(v,u)\}$. The directed graph $I$ is acyclic by assumption, and the directed graph $I'$ may fail to be acyclic --- in this case $J$ is the boundary ridge. This proves the second part.
\end{proof}

The simplicial complex $\DisDS_n$ does not have good properties of this sort. However, a description of its maximal simplices will be used in the following.

\begin{con}\label{conMaxSimpDisDS}
Let $(A_1,A_2)$ be a linearly ordered subdivision of the set $[n]$ into nonempty subsets, which means that $A_1\cup A_2=[n]$, $A_1\cap A_2=\varnothing$, $A_1,A_2\neq\varnothing$. Consider the following subset of $[n]^{(2)}$:
\[
I_{A_1,A_2}=A_1^{(2)}\sqcup A_2^{(2)}\sqcup \{(i_1,i_2)\mid i_1\in A_1, i_2\in A_2\}.
\]
In the digraph $\Gamma_{A_1,A_2}=([n],I_{A_1,A_2})$, all pairs of vertices from either $A_1$ or $A_2$ are connected in both directions, and there exists an arrow from any vertex of $A_1$ to any vertex of $A_2$, but not in the opposite direction. It can be seen that any such digraph $\Gamma_{A_1,A_2}$ satisfies the following.
\begin{itemize}
  \item It is not strongly connected (because there is no directed path from a vertex of $A_2$ to a vertex of $A_1$);
  \item It is maximal with respect to this property (if we add any arrow from $A_2$ to $A_1$, the graph becomes strongly connected).
\end{itemize}
Therefore $I_{A_1,A_2}$ is a maximal simplex of $\DisDS_n$. It can be seen that all maximal simplices of $\DisDS_n$ have the form $I_{A_1,A_2}$ for some ordered subdivision $(A_1,A_2)$ of $[n]$.
\end{con}

Since cardinality of $I_{A_1,A_2}$ equals $n(n-1)-|A_1||A_2|$, the maximal simplices have different dimensions when $n\geqslant 3$. Therefore $\DisDS_n$ is not pure.

Now we switch to the homotopy types of $\DAG_n$ and $\DisDS_n$. The following proposition is the main result of~\cite{BjWel}.

\begin{prop}\label{propBjWel}
Assume $n\geqslant 2$.
\begin{enumerate}
  \item The simplicial complex $\DAG_n$ is homotopy equivalent to $S^{n-2}$
  \item The simplicial complex $\DisDS_n$ is homotopy equivalent to the wedge of $(n-1)!$ many $(2n-4)$-dimensional spheres.
\end{enumerate}
\end{prop}

Let us show how this statement follows from Theorems~\ref{thmConstelSphere} and~\ref{thmBigConstelSphere}.

\begin{con}
Let $e_1,\ldots,e_n$ be the standard basis of the space $\Ro^n$. For any ordered pair $(i,j)\in [n]^{(2)}$ consider the vector $\alpha_{ij}=e_j-e_i$. The collection
\begin{equation}\label{eqAsystem}
A_{n-1}=\{\alpha_{ij}\mid (i,j)\in [n]^{(2)}\}
\end{equation}
is called \emph{a root system of type A}. It lies in the vector space of dimension $n-1$, namely in the hyperplane $\Pi=\left\{\sum x_i=0\right\}$. By scaling the norm in $\Ro^n$ we may assume that $A_{n-1}$ lies on the unit sphere $\Ss^{n-2}\subset\Pi\cong\Ro^{n-1}$.
\end{con}

\begin{figure}[h]
\begin{center}
\includegraphics[scale=0.4]{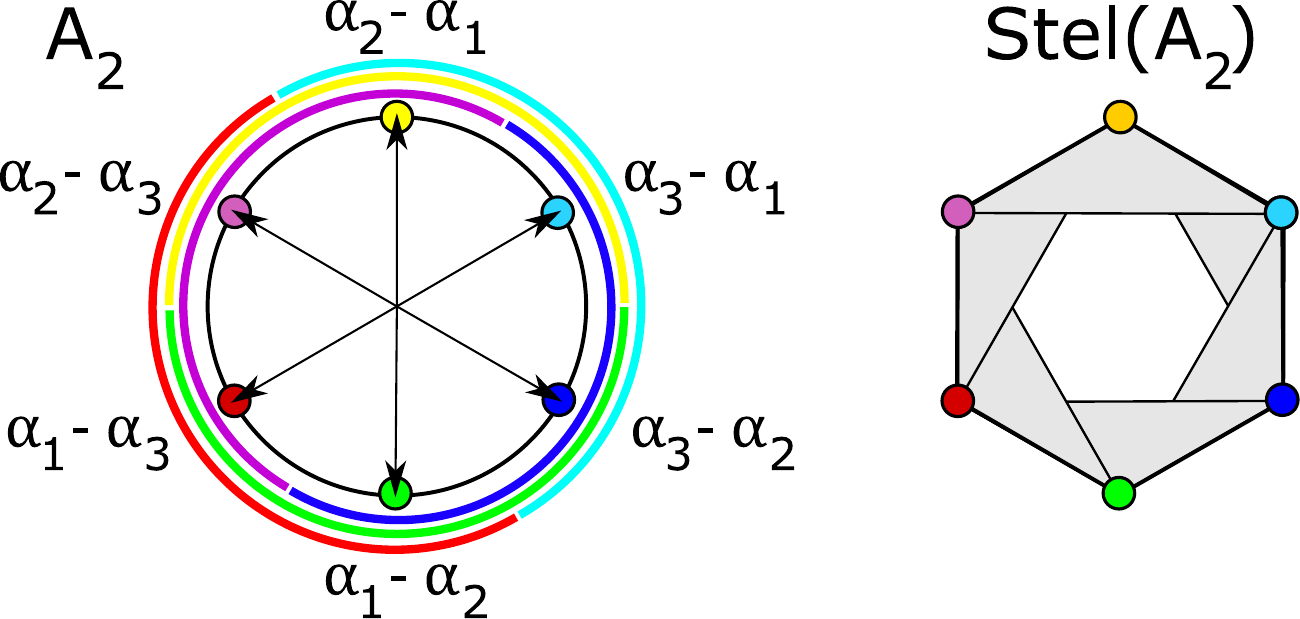}
\end{center}
\caption{The spherical configuration $A_2$ and the corresponding constellation complex. The covering of the unit circle $\Ss^1$ by hemispheres is shown on the left.}\label{figConstelA2}
\end{figure}

\begin{lem}\label{lemDAGisStel}
The simplicial complex $\DAG_n$ is naturally identified with the constellation complex $\Stel(A_{n-1})$ of the root system of type A.
\end{lem}

\begin{proof}
Notice that the open hemisphere $D(\alpha_{ij})$ is the intersection of the unit sphere $\Ss^{n-2}$ with the open half-space
\begin{equation}\label{eqOpHemisphereij}
P_{ij}=\{\langle x,\alpha_{ij}\rangle>0\}=\{x=(x_1,\ldots,x_n)\mid x_j>x_i\}.
\end{equation}
We will show two inclusions $\DAG_n\subseteq \Stel(A_{n-1})$ and $\Stel(A_{n-1})\subseteq\DAG_n$.

Consider a total DAG $I$ on the set $[n]$ (i.e. a maximal simplex of $\DAG_n$). It corresponds to some total order $\prec$ on $[n]$. This total order can be identified with the permutation $\sigma\colon[n]\to[n]$ such that $\sigma(i)\prec\sigma(j)$ if and only if $i<j$ in the natural order on $\{1,\ldots,n\}$. Recalling that $I$ is a subset of $[n]^{(2)}$, we need to prove that hemispheres from the set
\begin{equation}\label{eqSetHemispheres}
\{D(\alpha_{ij})\mid (i,j)\in I\} = \{D(\alpha_{ij})\mid \sigma(i)<\sigma(j)\}
\end{equation}
have nonempty common intersection. For any sequence of real numbers $b_1<\cdots<b_n$ with $\sum b_i=0$, the vector $b_\sigma=(b_{\sigma(1)},\ldots,b_{\sigma(n)})$ lies in the half-spaces $P_{ij}$ for any
\[
(i,j)\in J \Leftrightarrow \sigma(i)\prec \sigma(j) \Leftrightarrow i<j.
\]
Therefore, its normalization $b_\sigma/\|b_\sigma\|$ belongs to all hemispheres from the set~\eqref{eqSetHemispheres}. This proves that $I$ is a simplex of $\Stel(A_{n-1})$.

Now we consider any simplex $J$ of $\Stel(A_{n-1})$ and prove that $J\in\DAG_n$. Assume on the contrary that the directed graph $([n],J)$ has an oriented cycle, say,
\[
i_1\to i_2\to\cdots\to i_s\to i_1.
\]
Then $0$ is the barycenter of the vectors $\alpha_{i_1i_2},\alpha_{i_2i_3},\ldots,\alpha_{i_{s-1}i_s},\alpha_{i_si_1}$. Since the origin lies in the convex hull of the points $\{\alpha_{ij}\mid (i,j)\in J\}$, the corresponding open hemispheres do not intersect. This contradicts to the initial assumption $J\in\Stel(A_{n-1})$.
\end{proof}

\begin{ex}
The root system $A_2$ is shown on Fig.~\ref{figConstelA2}, left, these are the vertices of a regular hexagon. It can be seen that the nerve of the covering by open hemicircles centered in $A_2$ is isomorphic to the simplicial complex $\DAG_3$ shown on Fig.~\ref{figDAG3}.
\end{ex}

\begin{lem}\label{lemDisDSisBStel}
The simplicial complex $\DisDS_n$ is naturally identified with the big constellation complex $\BStel(A_{n-1})$ of the root system of type A.
\end{lem}

\begin{proof}
Similar to the proof of Lemma~\ref{lemDAGisStel}, the closed hemisphere $\bD(\alpha_{ij})$ is the intersection of the unit sphere $\Ss^{n-2}$ with the closed half-space
\[
\bP_{ij}=\{\langle x,\alpha_{ij}\rangle\geqslant0\}=\{x_j\geqslant x_i\}.
\]
We establish two inclusions $\DisDS_n\subseteq \BStel(A_{n-1})$ and $\BStel(A_{n-1})\subseteq\DisDS_n$.

Consider a maximal simplex $I_{A_1,A_2}$ of $\DisDS_n$ for some ordered subdivision $(A_1,A_2)$ of $[n]$ as defined in Construction~\ref{conMaxSimpDisDS}. There exist two numbers $\lambda_1<\lambda_2$ and a vector $b=(b_1,\ldots,b_n)\in\Ro^n$ with the properties
\begin{itemize}
  \item $b_i=\lambda_1$ for $i\in A_1$ and $b_i=\lambda_2$ for $i\in A_2$;
  \item $\sum_{i=1}^{n}b_i=0$ so that $b/\|b\|\in\Ss^{n-2}$.
\end{itemize}
Indeed, one can take $\lambda_1=-|A_2|$ and $\lambda_2=|A_1|$. It can be seen from the definition that the vector $b$ belongs to all closed halfspaces $\bP_{ij}$ for $(i,j)\in I_{A_1,A_2}$. Therefore, the intersection $\bigcup_{(i,j)\in I_{A_1,A_2}}\bD(\alpha_{ij})$ is nonempty. This proves the inclusion $\DisDS_n\subseteq \BStel(A_{n-1})$.

Now we consider any simplex $J$ of $\BStel(A_{n-1})$ so that there exists a point $x=(x_1,\ldots,x_n)\in\Ss^{n-2}$ in the intersection of all $\bP_{ij}$ for $(i,j)\in J$. To prove that $J\in\DisDS_n$ we assume on the contrary that the directed graph $([n],J)$ is strongly connected. Existence of a directed edge $(i,j)\in J$ implies the inequality $x_j\geqslant x_i$. Therefore, whenever there is a directed path from $i$ to $j$ we also have $x_j\geqslant x_i$ by transitivity. Strong connectivity then implies that all $x_i$ should be equal. This contradicts to the conditions $\sum_{i=1}^n x_i=0$ and $x\neq 0$.
\end{proof}

It is easily checked that $A_{n-1}$ is an ample and antipodal spherical configuration for $n\geqslant2$. Therefore Theorems~\ref{thmConstelSphere} and~\ref{thmBigConstelSphere} imply items 1 and 2 of Proposition~\ref{propBjWel} respectively.

\section{The poset of posets and the topology of topologies}\label{secOrdOrdersTopologies}

\subsection{Orders and preorders}\label{subsecOrdPreord}

Let $\Pos_n$ denote the set of all partial orders on the set $X$ of cardinality $n$. There is a natural partial order on $\Pos_n$ given by inclusion of relations. More precisely, we say that $r_1\preccurlyeq r_2$ for $r_1,r_2\in\Pos_n$ if $xr_1y$ implies $xr_2y$ for any pair $x,y\in X$. The poset of all posets $\Pos_n$ has the least element $\hat{0}$, --- the trivial order (that is the order in which any two distinct elements of $X$ are incomparable). We denote $\Pos_n\setminus\{\hat{0}\}$ by $\Pos_n^*$.

\begin{prop}[{\cite[Thm.5.5]{Bouc}}]\label{propBouc}
The geometrical realization $|\Pos_n^*|$ is homotopy equivalent to $S^{n-2}$.
\end{prop}

This statement follows from Theorem~\ref{thmConstelSphere}. There is a classical Galois correspondence between the poset $\DAG_n$ of directed acyclic graphs and the poset $\Pos_n$ of partial orders.

\begin{con}\label{conGaloisOrders}
For any DAG $\Gamma$ on the set $X$ consider the partial order $\leqslant_\Gamma$ on $X$ defined as the transitive closure of $\Gamma$. This means that $x\leqslant_\Gamma y$ if and only if there is a directed path from $x$ to $y$ in $\Gamma$ (probably, a trivial path of length $0$). The map $p\colon\DAG_n\to \Pos_n$ given by $p(\Gamma)=\leqslant_\Gamma$ is a morphism of posets.

For any partial order $\leqslant$ on $X$ consider the directed graph $\Gamma_\leqslant$ on $X$ naturally determined by this order. Namely, there is a directed edge from $x$ to $y$ if $x\neq y$ and $x\leqslant y$. The map $d\colon\Pos_n\to \DAG_n$ given by $d(\leqslant)=\Gamma_\leqslant$ is a morphism of posets.

It is easily observed that $p(d(\leqslant))=\leqslant$ for any $\leqslant\in\Pos_n$ and $d(p(\gamma))\geqslant_{\DAG_n}\gamma$ for any $\gamma\in\DAG_n$. Hence $d\colon\Pos_n\rightleftarrows\DAG_n\colon p$ is a Galois connection. Moreover, since $p\circ d=\id_{\Pos_n}$, we have a Galois insertion: the poset $\Pos_n$ can be considered a subposet in $\DAG_n$.

Notice that a nontrivial order always induces a nontrivial graph, and vice versa. Therefore there is also a Galois correspondence
\begin{equation}\label{eqPosDags}
d\colon\Pos_n^*\rightleftarrows\DAG_n\setminus\{\varnothing\}\colon p
\end{equation}
\end{con}

Proposition~\ref{propBouc} now follows from Corollary~\ref{corGalois} and item 1 of Proposition~\ref{propBjWel} (notice that the geometrical realization of any simplicial complex $K$ is homeomorphic to the geometrical realizaton of the poset $K\setminus\{\varnothing\}$ of its nonempty simplices).

In a similar fashion, one can study the poset of all preorders. Let $\PrePos_n$ denote the set of all preorders on a set of cardinality $n$, ordered by inclusion of relations. As before, $\hat{0}$ denotes the empty preorder (the least element of $\PrePos_n$), and $\hat{1}$ denotes the full preorder in which all pairs are comparable (the greatest element of $\PrePos_n$). Let $\PrePos_n^*=\PrePos_n\setminus\{\hat{0},\hat{1}\}$. The next statement appears in~\cite{BoucTalk} without a proof.

\begin{prop}[\cite{BoucTalk}]\label{propBoucTalk}
The geometrical realization $|\PrePos_n^*|$ is homotopy equivalent to $S^{n-2}$.
\end{prop}

\begin{con}\label{conGaloisPreorders}
Again, we construct a Galois correspondence between proper preorders on $X$ and certain digraphs on $X$. Any preorder is a naturally a digraph. On the other hand, any digraph can be complemented to a preorder by taking its transitive closure. Notice that a graph is strongly connected if and only if its transitive closure is the full preorder $\hat{1}\in\PrePos_n$. Therefore, we also have a Galois correspondence

\begin{equation}\label{eqPrePosDisDS}
d\colon\PrePos_n^*\rightleftarrows\DisDS_n\setminus\{\varnothing\}\colon p.
\end{equation}
\end{con}

Proposition~\ref{propBoucTalk} follows from Corollary~\ref{corGalois} and item 2 of Proposition~\ref{propBjWel}.

\subsection{Topology of topologies on a finite set}\label{subsecTopTops}

For a finite set $X$ of cardinality $n$ consider the set $\Top_n$ of all topologies on $X$. This set is partially ordered by the strength of topology: we say that $\ca{T}_1\leqslant\ca{T}_2$ if any open subset of the topology $\ca{T}_1$ is also an open subset of the topology $\ca{T}_2$. The poset $\Top_n$ has the greatest element $\hat{1}$ (the discrete topology on $X$), and the least element $\hat{0}$ (the indiscrete topology on $X$). Notice that $\Top_n$ is actually a lattice~\cite{Steiner,LarAnd}. However, it is not a modular lattice for $n\geqslant 3$ \cite[Th.3.1]{Steiner}, in particular it is not geometric, so Theorem~\ref{thmBjorner} is not applicable to describe the homotopy type of this lattice.

Recall that topology $\ca{T}$ is said to satisfy Kolmogorov separation axiom $T_0$ if, for any two distinct points $x,y\in X$, $x\neq y$ there exists either an open neighborhood $U$ of $x$ such that $y\notin U$ or an open neighborhood $U$ of $y$ such that $x\notin U$. All $T_0$-topologies form a subposet of $\Top_n$ which we denote by $\TopT0_n$. Notice that $\hat{1}\in\TopT0_n$ while $\hat{0}\notin\TopT0_n$.

Results of the previous subsection imply the following proposition.

\begin{prop}\label{propTopologies}
\begin{enumerate}
  \item The geometrical realization of the poset $\Top_n\setminus\{\hat{0},\hat{1}\}$ is homotopy equivalent to the wedge of $(n-1)!$ many $(2n-4)$-dimensional spheres.
  \item The geometrical realization of the poset $\TopT0_n\setminus\{\hat{1}\}$ is homotopy equivalent to $S^{n-2}$.
\end{enumerate}
\end{prop}

This is a straightforward consequence of the basic correspondence between preorders and Alexandrov topologies on a given set. However, we couldn't find the result of Proposition~\ref{propTopologies} explicitly stated in the existing literature, so we included the explanation for the completeness of exposition.

\begin{lem}\label{lemTopPosets}
There are isomorphisms of posets $\PrePos_n\cong\Top_n$ and $\Pos_n\cong\TopT0_n$.
\end{lem}

\begin{proof}
A preorder $\leqslant$ on $X$ determines an Alexandrov topology $T_{\leqslant}$ whose open sets are all upper order ideals:
\[
U\in T_{\leqslant} \text{ if and only if } x\in U \text{ and } x\leqslant y \text{ imply } y\in U.
\]
On the other hand, given a topology $\ca{T}$ on a finite set $X$, a preorder $\leqslant_{\ca{T}}$ can be constructed as follows. For any $x\in X$, consider the least open neighborhood $U_x$ of $x$, which can be defined as the intersection of all neighborhoods of $x$. Define the preorder
\[
x\leqslant_{\ca{T}}y \text{ if and only if } U_x\supseteq U_y.
\]
It can be easily checked that two described constructions are inverses of one another, and they preserve the natural inclusion orders on $\PrePos_n$ and $\Top_n$.

For any $T_0$-topology $\ca{T}$ on $X$ and any two points $x\neq y$, either $U_x$ does not contain $y$ or $U_y$ does not contain $x$. Therefore, the relations $x\leqslant_{\ca{T}} y$ and $y\leqslant_{\ca{T}}x$ cannot hold simultaneously. This shows that  $\leqslant_{\ca{T}}$ is an order. Conversely, if $\leqslant$ is an order, then, given $x\leqslant y$ and $x\neq y$, we have that the neighborhood $U_y$ does not contain $x$. Hence $\ca{T}_{\leqslant}$ is a $T_0$-topology.
\end{proof}

In view of Lemma~\ref{lemTopPosets}, Proposition~\ref{propTopologies} is just a reformulation of Propositions~\ref{propBouc} and~\ref{propBoucTalk}.

\section{The Gale dual polytope of DAGs}\label{secGale}

\subsection{Gale duality}

In this section we recall the basic correspondence between convex polytopes and their affine Gale diagrams, and introduce a convex polytope of cycles $\Pc_n$ which is in certain sense dual to the space $\DAG_n$. Affine Gale duality is the classical subject in convex geometry. For a good exposition of this subject we refer to the book of Gr\"{u}nbaum~\cite{Gr} which contains all statements necessary for our context.

Let us call a finite multi-set $Z$ of points in $\Ss^{d-1}\sqcup\{0\}$ \emph{an extended spherical configuration}. Notice that any vector of $\Ro^d$ can be normalized to a point in $\Ss^{d-1}\sqcup\{0\}$. An extended spherical configuration $Z$ will be called \emph{doubly ample} if every open hemisphere of $\Ss^{d-1}$ contains at least $2$ points of $Z$, with multiplicities taken into account. This means that open hemispheres centered in non-zero points of $Z$ wrap the sphere at least twice.

With any $r$-dimensional convex polytope $P$ on $m$ vertices, one can associate its affine Gale diagram $G(P)$, which is a doubly ample extended spherical configuration of $m$ points in $\Ss^{m-r-2}\sqcup\{0\}$.

\begin{thm}[{Affine Gale duality, \cite[Sec.5.4]{Gr}}]\label{thmGaleMain}
There is a one-to-one correspondence between two classes of objects
\begin{itemize}
  \item $r$-dimensional convex polytopes with $m$ vertices in $\Ro^r$ up to affine transformations;
  \item doubly ample extended spherical configurations of $m$ points in $\Ss^{m-r-2}\sqcup\{0\}$ up to affine transformations and normalizations.
\end{itemize}
The correspondence from polytopes to spherical diagrams is given by affine Gale diagrams.
\end{thm}

The combinatorial structure of a polytope $P$ can be read of its Gale diagram as follows.

\begin{prop}[{\cite[Sec.5.4(1)]{Gr}}]\label{propGaleMain}
Vertices $i_1,\ldots,i_k$ belong to one proper face of $P$ if and only if the points of $G(P)$ corresponding to the complement $[m]\setminus\{i_1,\ldots,i_k\}$ contain $0$ in their convex hull.
\end{prop}

\begin{cor}
The set $I=\{i_1,\ldots,i_k\}$ is the vertex set of some facet of $P$ if and only if the complement $[m]\setminus I$ is a minimal non-simplex of $\Stel(G(P))$.
\end{cor}

\subsection{Combinatorial Alexander duality}

Proposition~\ref{propGaleMain} can be conceptualized in two natural steps.

It is not difficult to see that whenever one knows which collections of vertices form facets of a polytope, the whole lattice of faces of $P$ can be reconstructed by taking intersections of facets. The first author developed this construction in the theory of nerve-complexes~\cite{AB}.

\begin{con}\label{conNerveCpx}
Let $P$ be a convex polytope of dimension $d$ with the vertex set $v_1,\ldots,v_m$. We define a simplicial complex $K_{P^*}$ on the set $[m]=\{1,\ldots,m\}$ by the condition that $\{i_1,\ldots,i_s\}\in K_{P^*}$ if and only if the vertices $v_{i_1},\ldots,v_{i_s}$ lie in a single facet. The complex $K_{P^*}$ is called \emph{the nerve-complex} of the polar dual polytope $P^*$. It can be easily seen that $K_{P^*}$ is always homotopy equivalent to $S^{d-1}$. If $P$ is simplicial, then $K_{P^*}$ is combinatorially equivalent to its boundary $\dd P$.
\end{con}

Recall the classical notion of the combinatorial Alexander duality.

\begin{con}\label{conAlexDual}
Let $K$ be a simplicial complex on the set $[m]$, which is neither the whole simplex $\Delta_{[m]}$ nor its boundary. \emph{The combinatorial Alexander dual complex} is defined by
\[
\hat{K}=\{I\subset[m]\mid [m]\setminus I\notin K\}.
\]
The principal result about this notion states that barycentric subdivisions of $K$ and $\hat{K}$ can be embedded in $\dd\Delta_{[m]}'\cong S^{n-2}$ as Alexander dual subsets. This implies the (ordinary) Alexander duality
\[
\Hr_{j}(\hat{K};R)\cong \Hr^{m-3-j}(K;R).
\]
In particular, if homology of $K$ are concentrated in degree $r$, then cohomology of $\hat{K}$ are concentrated in degree $m-3-r$, and the complexes have the same top Betti number.
\end{con}

In terms of Constructions~\ref{conNerveCpx} and~\ref{conAlexDual}, Proposition~\ref{propGaleMain} can be reformulated as follows.

\begin{cor}
Let $P$ be a convex polytope. Then the constellation complex of the Gale dual configuration $G(P)$ coincides with the combinatorial Alexander dual complex to the nerve complex $K_{P^*}$:
\[
\Stel(G(P))=\widehat{K_{P^*}}.
\]
\end{cor}

\begin{rem}\label{remAmpleToPoly}
It follows from Theorem~\ref{thmGaleMain} that, for any doubly ample spherical configuration $Z$, there exists a polytope $P$ such that $\Stel(Z)$ is combinatorial Alexander dual to $K_{P^*}$.
\end{rem}

\subsection{Polytope of cycles}\label{subsecPolyCycles}

The application of Remark~\ref{remAmpleToPoly} to the root system of type $A$ seems a natural thing one can do.

\begin{con}\label{conCyclePolytope}
Notice that the type A spherical configuration $A_{n-1}\subset\Ss^{n-2}$ is doubly ample for $n\geqslant 3$. Indeed, if $x=(x_1,\ldots,x_n)\in\Pi=\{\sum x_i=0\}$ is a nonzero vector, then not all of $x_i$'s are equal. Therefore, if $n\geqslant 3$ there are at least two strict inequalities of the form $x_i>x_j$ on the coordinates of this vector, which means that the point $x$ lies in at least two open hemispheres centered at $\alpha_{ij}$.

According to Remark~\ref{remAmpleToPoly}, there exists a convex polytope $\Pc_n$ of dimension $n(n-1)-(n-2)-2=n(n-2)$ with $n(n-1)$ many vertices which is Gale dual to the configuration $A_{n-1}$. In particular, we have $\widehat{K}_{\Pc_n^*}=\Stel(A_{n-1})$. The facets of $\Pc_n$ are the complements to minimal non-simplices of $\Stel(A_{n-1})=\DAG_n$. Non-simplices of $\DAG_n$ are the directed graphs which have oriented cycles. Therefore minimal non-simplices are oriented cycles themselves. The vertex set of any facet of $\Pc_n$ is therefore the complement to some oriented cycle.
\end{con}

\begin{ex}
For $n=3$, the Gale dual $\Pc_3$ polytope to the configuration $A_2$ is a 3-dimensional triangular prism shown on Fig.~\ref{figCDGpolytope}. It has $6$ vertices labelled by ordered pairs $(i,j)$, $i\neq j$ (we denote the vertices by $\beta_{ij}$ to distinguish them from the elements $\alpha_{ij}$ of $A_2$). The facets correspond to complements of directed cycles. There exist $5$ oriented cycles on $3$ vertices. Three cycles of length $2$ have the form $1\to 2\to 1$, $1\to 3\to 1$, $2\to 3\to 2$: their complements form quadrangular side faces of the prism. The two cycles of length $3$, $1\to 2\to 3\to 1$ and $3\to 2\to 1\to 3$, are complementary to each other, they form the prism bases.
\end{ex}

\begin{figure}[h]
\begin{center}
\includegraphics[scale=0.5]{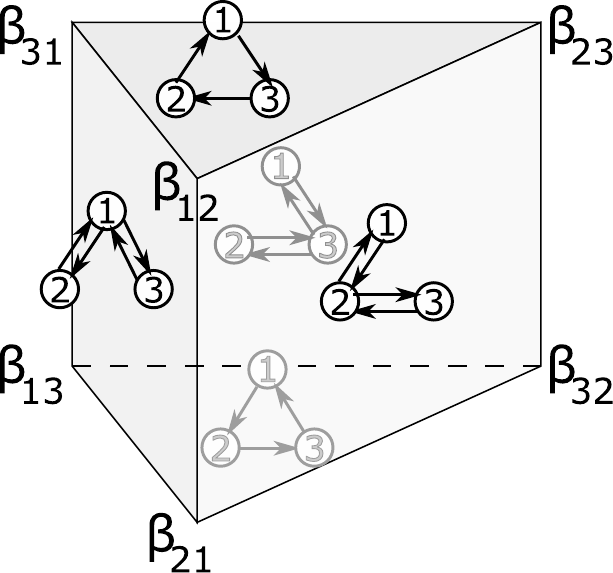}
\end{center}
\caption{The polytope $\Pc_3$ Gale dual to the configuration $A_2$.}\label{figCDGpolytope}
\end{figure}

The next statement easily follows from the properties of Gale duality and the definition of the polytope $\Pc_n$.

\begin{prop}\label{propFacesCycles}
Consider the set $[n]^{(2)}$ of all ordered pairs $(i,j)$, $i,j\in[n]$, $i\neq j$, and let $\Bb([n]^{(2)})$ be the boolean lattice on this set. Consider the subset $\Cy\subset\Bb([n]^{(2)})$ of all cycles on $n$ vertices. Then the (semi)lattice $\Jc_n$ formed by the unions of elements of $\Cy$ in $\Bb([n]^{(2)})$ is isomorphic to the lattice of faces of a convex $n(n-2)$-dimensional polytope. In particular, the geometrical realization $|\Jc_n\setminus\{\hat{0},\hat{1}\}|$ is homeomorphic to $(n^2-2n-1)$-dimensional sphere.
\end{prop}

\begin{proof}
Instead of a set $A\subset [n]^{(2)}$, take its complement. Then we can identify $\Jc_n$ with the (semi)lattice of $\Bb([n]^{(2)})$ formed by intersections of the complements to cycles, up to order reversal. However, all possible intersections of complements to cycles correspond to proper faces of the polytope $\Pc_n$ according to Construction~\ref{conCyclePolytope}. Remembering the order reversal, we see that $\Jc_n$ is isomorphic to the face poset of the polar dual polytope $\Pc_n^*$ of dimension $n(n-2)$.
\end{proof}

\begin{rem}
Although the construction of Gale duality is explicit and constructible for each particular $n$, we are unaware of any uniform description of a convex realization of either $\Pc_n$ or its polar dual.
\end{rem}

\section{Proposed applications and remarks}\label{secSpherDAGs}

\subsection{Nerves beyond nerve theorem}\label{subsecNerves}
In this subsection we remind the reader the basic setting of topological data analysis and explain the relation of some constructions to the current work.

\begin{defin}
A simplicial filtration on a finite set $[m]$ is a collection $\F=\{K_t\mid t\in\Rg\}$ such that
\[
K_{t_1}\subseteq K_{t_2}\text{ for }t_1<t_2.
\]
\end{defin}

It is usually assumed that $K_{+\infty}=\bigcup_tK_t$ coincides with the full simplex on the vertex set $[n]$. Every subset $I\subset [n]$ appears in a filtration $\F$ at the moment $\inf\{t\mid I\in K_t\}$. Therefore, a simplicial filtration can be alternatively encoded by a discrete function $\birth\colon 2^{[n]}\to\Rg$ satisfying the property
\[
I\subset J \text{ implies }\birth(I)\leqslant \birth(J).
\]
This function encodes the birth times of the simplices.

Persistent homology is an algebraical invariant which allows to describe the temporal dynamics of simplicial homology of a filtration. We refer to ?? for the details of this construction.

\begin{con}\label{conCzechVR}
Given a metric space $(M,d)$ together with a finite point cloud $X=\{x_1,\ldots,x_m\}\subset M$ one can define two common types of simplicial filtrations:
\begin{itemize}
  \item Vietoris--Rips filtration $\F_{VR}(M,d,X)=\{K_t^{VR}\}$. A simplex $I\subset [m]$ of this filtration is born at the time moment
  \[
  \frac12\min\{d(x_i,x_j)\mid \{i,j\}\in I, i\neq j\}.
  \]
  In other words, a simplex $I$ lives at the time moment $t\in\Rg$ if and only if all pairwise distances between its vertices exceed $t$. It can be seen that the whole ambient metric space $M$ is not needed to construct Vietoris--Rips filtration: $\F_{VR}(M,d,X)=\F_{VR}(X,d,X)$.
  \item \v{C}ech filtration $\F_{C}(M,d,X)=\{K_t^{C}\}$. A simplex $I\subset [m]$ of this filtration is born at the time moment
  \[
  \inf\left\{t\in\Rg\mid \bigcap\nolimits_{i\in I}\bar{B}_{t}(x_i)\neq\varnothing\right\},
  \]
  where $\bar{B}_{r}(x)$ denotes the closed ball of radius $r$ centered in $x$, in the metric space $M$. This definition essentially require the whole metric space $M$.
\end{itemize}
\end{con}

The underlying idea of topological data analysis consists in the belief that persistent homology of either $\F_{VR}(M,d,X)$ or $\F_{C}(M,d,X)$ correlate with the actual homology of the metric space $M$, for sufficiently dense data clouds $X$. This belief has the following grounds:

\begin{enumerate}
  \item The term $K_t$ of the \v{C}ech filtration coincides with the nerve of the covering $\bigcup_{i\in[m]}\bar{B}_{t}(x_i)$. If $t$ is sufficiently large, then the union of balls covers the whole space $M$, and, provided that the intersections of the covering are contractible, we are in position to apply the nerve theorem and get homotopy equivalence $K_t\simeq M$.
  \item Vietoris--Rips and \v{C}ech filtrations are interconnected $K_t^{C}\subseteq K_t^{VR}\subseteq K_{2t}^C$, so, in certain sense, topological contents of these two filtrations are ``statistically'' the same.
  \item There exist a bunch of persistence theorems which relate persistent homology of various data clouds $X$ sampled from the same space $M$.
\end{enumerate}

See???

All items stated above should be approached with certain criticism. First of all, we give a somehow elementary example, which demonstrates that Vietoris--Rips filtration may produce counter-intuitive results.

\begin{ex}\label{exCube}
Consider the set of vertices of the $n$-dimensional cube in $\Ro^n$:
\[
X=\{(\varepsilon_1,\ldots,\varepsilon_n)\mid \varepsilon_i=\pm1\},
\]
with the standard Euclidean metric. Then, at the time interval $E=[\sqrt{n-1},\sqrt{n})$ the Vietoris--Rips filtration has a persistent homology of degree $2^{n-1}-1$. Indeed, when $t\in E$, the complex $K_t^{VR}(X)$ is combinatorially equivalent to the simplicial complex
\[
\{i_1,i'_1\}\ast \{i_1,i'_1\}\ast\cdots\ast \{i_{2^{n-1}},i'_{2^{n-1}}\}\cong S^{2^{n-1}-1},
\]
where the pairs $i_j,i'_j$ correspond to the endpoints of main diagonals of the cube.
\end{ex}

Intuitively, we expect that a cube in a Euclidean space should have some simple homology, however, this simple calculation shows this is not the case for Vietoris--Rips filtration. Certainly, this example does not contradict to the known results, since the lifespan of this huge-degree homology equals $\sqrt{n}-\sqrt{n-1}$, which is small for large $n$. However, the example shows that degrees of persistent homology can be very large being compared to the dimension of the ambient space.

One can see that persistent homology of mathematically structured data is sensitive to combinatorics not just topology. Moreover, this combinatorial information about the data set inhabits large homological degrees. We believe that this phenomenon should be investigated further in more detail.

Our next example shows the situation, when computation of persistent homology produces counterintuitive answers, even for the \v{C}ech filtration.

\begin{ex}\label{exSphere}
Consider the metric space $\Ss^{d-1}$, the round sphere of unit radius with the geodesic metric $\delta$. Let $Z=\{z_{\pm1},\ldots,z_{\pm m}\}\subset\Ss^{d-1}$ be an ample antipodal spherical configuration as defined in Section~\ref{secIntro}. Consider the \v{C}ech filtration $\F_{C}(\Ss^{d-1},\delta,Z)=\{K_t^{C}\}$. Since the open hemispheres $D(z_i)$ cover the whole sphere, the filtration terms $K_{t}^{C}$ are homotopy equivalent to $S^{d-1}$ for $t=\pi/2-\varepsilon$ and sufficiently small $\varepsilon>0$. Indeed, in this range the nerve theorem is applicable.

However, at the time moment $t_1=\pi/2$, the simplicial complex $K_{t_1}^{C}=\BStel(Z)$ becomes homotopy equivalent to the wedge $\bigvee_{\mu(\Ha)}S^{2d-2}$ according to Theorem~\ref{thmBigConstelSphere}. This means that $\mu(\Ha)$ persistent homology of degree $2d-2$ are born at the time moment $t_1$. 

We don't have an estimation of the life durations of these persistent homology, however, their number $\mu(\Ha)$ may be very large. Indeed, even in the case of $Z=A_d$, the number of persistent homology of degree $2d-2$ is equal $d!$ which is very large compared both to the dimension $d-1$ of the metric space and the number $d(d+1)$ of points in the point cloud.
\end{ex}

Again, Example~\ref{exSphere} does not contain any actual contradiction with established mathematical results, since the balls in general metric spaces are not expected to satisfy the assumption of the nerve theorem. However, this example shows that the work with persistent homology, even on smooth manifolds, should be made with certain care. Another example of this sort appeared in our work~??, where we computed 3-dimensional persistent homology of a point cloud sampled, in a regular fashion, from a thickened torus $T^2\times D^1$, moreover, the lifetime of this parasite homology was equal to the lifetimes of the meaningful features.

At time scales which are large enough compared to the inner geometrical features of the manifold, some high degree homological features may appear which do not highlight any topological properties. Ultimately, this observation is an evidence against usage of persistent homology in machine learning, where common topics, such as manifold conjecture, usually deal with extremely high dimensions.

%\begin{rem}
%Notice that, for any antipodal ample configuration $Z=\{z_{\pm1},\ldots,z_{\pm m}\}\subset\Ss^{d-1}$ a reasoning similar to Example~\ref{exCube} can be performed. Consider the number $t_2=\pi$. Notice that the terms $K_{t}^{C}$ of the \v{C}ech filtration become the simplices at time moments $t\geqslant t_2$. However, the terms $K_{t_2-\varepsilon}^{C}$, for sufficiently small $\varepsilon>0$, are homotopy equivalent to the sphere of dimension $2m-1$. Indeed, for such time moments, the whole simplex
%\end{rem}

\subsection{Spherical representation of DAGs}\label{subsecSpherRepresent}
The idea used in the proof of Proposition~\ref{propBjWel} can be adopted for the spherical encoding of DAGs.

\begin{con}\label{conFromDAGtoVector}
Assume we are given a DAG with nonnegative weights attached to the edges. This means we are given a triple $\Gamma=([n],E,w)$, where the edge-set $E\in\DAG_n=\Stel(A_{n-1})$ and $w=(w_{ij}\mid (i,j)\in E)$ is the list of nonnegative real weights attached to the edges of $\Gamma$. We assume that a zero weight of an edge corresponds to the situation when the edge is absent from a graph. Let $\uDAG_n$ denote the space of all weighted directed graphs whose weights sum to $1$. Then $\uDAG_n$ is naturally homeomorphic to the geometrical realization of the complex $\DAG_n$, and the space of all weighted directed graphs is an infinite cone over $\uDAG_n$.

By the definition of $\uDAG_n$, we have $\|w\|=1$, so $w_{ij}$ do not vanish simultaneously. Then the weighted DAG $\Gamma$ can be represented on a sphere by the normalization $x_\Gamma/\|x_\Gamma\|$ of the vector
\[
x_\Gamma=\sum\nolimits_{(i,j)\in E} w_{ij}\alpha_{ij}\in\Pi\subset\Ro^n
\]
(this vector is nonzero since the vectors $\{\alpha_{ij}\mid (i,j)\in E\}$ lie in an open halfspace).
\end{con}

This correspondence is not one-to-one: many different weighted graphs represent one point on a sphere (this happens because $\uDAG_n$ is not homeomorphic to $S^{n-2}$, just homotopy equivalent). The ways of choosing a unique weighted DAG for a point on a sphere depend on the particular task under consideration. One of the ways to restore a graph from a vector on a sphere is the following.

\begin{con}\label{conFromVectorToDAG}
Take a vector $x=(x_1,\ldots,x_n)\in\Ss^{n-2}\subset\Pi$. Since $x\neq0$ and $\sum x_i=0$ there exist at least one ordered pair $(i,j)$ such that $x_i<x_j$. The graph can be constructed as follows. For any pair $(i,j)$ such that $x_i<x_j$ add an edge $(i,j)$ to a graph and enhance it with the weight $x_j-x_i>0$.

A variation of this construction starts with arbitrary differentiable function $\rho\colon \Rg\to \Rg$ such that $\rho(0)=0$. Then, whenever $x_i<x_j$ we add the edge $(i,j)$ to a graph weighted with $\rho(x_j-x_i)$.
\end{con}

\begin{rem}
Be aware that Construction~\ref{conFromVectorToDAG} is neither left nor right inverse to the Construction~\ref{conFromDAGtoVector}.
\end{rem}

\begin{rem}
We propose the following application of the space $\DAG_n$. Many classical problems of machine learning are reformulated as the problem of minimization of some smooth (or piecewise smooth) function $L\colon X\to\Ro$, called the loss function. The domain $X$ is usually a Euclidean space $\Ro^w$ of network weights, and the classical (or stochastic) gradient descent algorithm works well for such problem. In the area of the automated machine learning (AutoML) one is able to vary the architecture of the network, not just its weights, in order to find the optimal network shape for the given set of tasks. The notion of the architecture is rather fuzzy. However, in many cases the neural network architecture is represented by a DAG with some additional coloring --- the names of operations at the vertices. Therefore, the optimization procedure on the space of all neural architectures seems related to the problem of mathematical description of this space itself.

The homotopy type of the space $\DAG_n$ is already described in the work~\cite{BjWel} in a quite constructive way (see Construction~\ref{conFromDAGtoVector} and~\ref{conFromVectorToDAG} above). If one needs to utilize not just the homotopy type of this space, but its topology, then Lemma~\ref{lemPseudoMfd} may appear important. Indeed, since $\DAG_n$ is a pseudomanifold with boundary, the gradient descent algorithm on this space makes a perfect sense. We propose to explore the experimental aspect of this problem in future research on specific neural architecture search (NAS) tasks.
\end{rem}

We end up the paper with the following meta-observation.

\begin{rem}
There is a variety of spaces encoding certain common properties of graphs and digraphs, of which the spaces $\DAG_n$ and $\DisDS_n$ are just particular representatives. Such spaces are studied in the evasiveness theory, which studies monotone properties of graphs and digraphs. For any monotone property $\ca{P}_n$ of (di)graphs on $n$ vertices, consider the space $K(\ca{P}_n)$ of all (digraphs) satisfying this property. The basic statement of the evasiveness theory asserts that if the space $K(\ca{P}_n)$ is non-contractible, then whether a graph $\Gamma$ has a property $\ca{P}_n$, cannot be checked without inspecting all possible edges of $\Gamma$ in general. A number of results about the topology of the spaces of properties have been obtained in the literature. 

Potentially, all these results may be applied in machine learning. If one needs to solve an optimization problem over a specific set of graphs, this task can be approached by continuous gradient methods. In particular, the total Betti number $\beta(K(\ca{P}_n))$ of a property $\Pa_n$ may serve as an estimate for the number of stationary points of the optimization process via Morse theory. See also~\cite{Forman} for a discrete Morse theory approach to the evasiveness.
\end{rem}

We expect that new problems in combinatorial topology may be motivated by graph properties originating in machine learning.


\begin{thebibliography}{99}

\bibitem{AyzConst}
A.\,A.\,Ayzenberg, \textit{Simplicial complexes Alexander dual to boundaries of polytopes}, 2013, preprint \href{https://arxiv.org/abs/1310.5487}{arXiv:1310.5487}.

\bibitem{AB} A.\,A.\,Ayzenberg, V.\,M.\,Buchstaber, \textit{Moment-angle spaces and nerve-complexes of convex polytopes}, Proceedings of the Steklov Institute of Mathematics, V.275, 2011.

\bibitem{AyzSubstit} A.\,A.\,Ayzenberg, \textit{Substitutions of polytopes and of simplicial complexes, and multigraded betti numbers}, Trans. Moscow Math. Soc. (2013), 175-202.

\bibitem{BBLSW}
E.\,Babson, A.\,Bj\"{o}rner, S.\,Linusson, J.\,Shareshian, V.\,Welker, \textit{Complexes of not $i$-connected graphs}, Topology 38 (1999), 271--299.

\bibitem{Barmak}
J.\,A.\,Barmak, \textit{On Quillen's theorem A for posets}, Journal of Combinatorial Theory Ser. A, 118:8 (2011), 2445--2453.

\bibitem{LBB} M.\,Best, P.\,V.\,E.\,Boas, H.\,W.\,Lenstra, \textit{A sharpened version of the Aanderaa-Rosenberg conjecture}, Report ZW 30/74, Mathematisch Centrum Amsterdam (1974).

\bibitem{Bj}
A.\,Bj\"{o}rner, M.\,L.\,Wachs, V.\,Welker, \textit{Poset fiber theorems}, Trans. Amer. Math. Soc. 357:5 (2005), 1877--1899.

\bibitem{BjMatr}
A.\,Bj\"{o}rner, \textit{Homology and Shellability of Matroids and Geometric Lattices}, in ``Matroid Applications'' ed. N.\,White, 1990.

\bibitem{BjWel}
A.\,Bj\"{o}rner, V.\,Welker, \textit{Complexes of Directed Graphs}, SIAM Journal on Discrete Mathematics 12:4 (1999), 413--424.

\bibitem{Bolobas}
B. Bollob\'{a}s, \textit{Complete subgraphs are elusive}, J. Combin. Theory Ser. B, 21 (1976), 1--7.


\bibitem{Bouc}
S.\,Bouc, \textit{The poset of posets}, 2013, preprint \href{https://arxiv.org/abs/1311.2219v1}{arXiv:1311.2219v1}.

\bibitem{BoucTalk}
S.\,Bouc, J.\,Th\'{e}venaz, \textit{The algebra of essential relations on a finite set}, talk at Third International Symposium on
Groups, Algebras, and Related Topics, Peking University 2013, slides are available \href{https://www.lamfa.u-picardie.fr/bouc/}{online}.

%\bibitem{PosPos}
%R.\,A.\,Brualdi, H.\,C.\,Jung, W.\,T.\,Trotter Jr., \textit{On the poset of all posets on $n$ elements}, Discrete Applied Mathematics 50 (1994), 111--123.

\bibitem{CKS}
A.\,Chakrabarti, S.\,Khot, Y.\,Shi, \textit{Evasiveness of Subgraph Containment and Related Properties}, In: A.\,Ferreira, H.\,Reichel, (eds) STACS 2001. Lecture Notes in Computer Science, 2010.

\bibitem{Folk}
J.\,Folkman, \textit{The homology group of a lattice}, J. Math. and Mech., 15 (1966), 631--636.

\bibitem{Forman}
R.\,Forman, \textit{Morse Theory and Evasiveness}, Combinatorica 20 (2000), 489--504.

\bibitem{Gr} B.\,Gr\"{u}nbaum. Convex Polytopes, 2nd ed. Graduate Texts in Mathematics Vol. 221, 2003.

%\bibitem{JunRin} M.\,Jungerman, G.\,Ringel, \textit{Minimal triangulations on orientable surfaces}, Acta Math. 145 (1980), 121--154.

\bibitem{KSS}
J.\,Kahn, M.\,Saks, D.\,Sturtevant, \textit{A topological approach to evasiveness}, Combinatorica, 4:4 (1984), 297--306.

\bibitem{Kook}
W.\,Kook, \textit{Categories of acyclic graphs and automorphisms of free groups}, Ph.D. thesis, Stanford University, 1996.

\bibitem{LarAnd}
R.\,E.\,Larson, S.\,J.\,Andima, \textit{The lattice of topologies: A survey}, Rocky Mountain J. Math. 5:2 (1975), 177--198.

\bibitem{Lutz}
F.\,H.\,Lutz, \textit{Some results related to the evasiveness conjecture}, Journal of Combinatorial Theory, Series B, 81:1 (2001), 110--124.

\bibitem{McCord}
M.\,C.\,McCord, \textit{Singular homology groups and homotopy groups of finite topological spaces}, Duke Math.J. 33 (1966), 465--474.

\bibitem{Quil}
D.\,Quillen, \textit{Higher algebraic K-theory, I: Higher K-theories}, Lecture Notes in Math. 341 (1973), 85--147.

\bibitem{RinYoungs}
G.\,Ringel, J.\,W.\,T.\,Youngs, \textit{Solution of the Heawood Map-Coloring Problem}, Proc. Nat. Acad. Sci. USA 60, 438--445 (1968).

\bibitem{Singh}
A.\,Singh, \textit{Higher matching complexes of complete graphs and complete bipartite graphs}, Discrete Mathematics 345:4 (2022).

\bibitem{Steiner}
A.\,K.\,Steiner, \textit{The Lattice of Topologies: Structure and Complementation}, Transactions of the AMS 122:2 (1966), 379--398.

\bibitem{Vas}
V.\,Vassiliev, \textit{Complexes of connected graphs}, in The Gelfand Mathematical Seminars, 1990--1992, Birkh\"{a}user Boston, 1993, 223--235.

\bibitem{Wachs}
M.\,L.\,Wachs, \textit{Topology of matching, chessboard, and general bounded degree graph complexes}, Algebra univers. 49 (2003), 345--385.


\end{thebibliography}
\end{document}